\documentclass[12pt,oneside,reqno]{amsart}

\usepackage{amsmath,amscd,amsthm,amsfonts,amssymb}
\usepackage{mathtools,mathrsfs}
\mathtoolsset{showonlyrefs}
\usepackage{comment}
\usepackage{todonotes}
\usepackage{bm}
\usepackage{bbm}
\usepackage{url}
\usepackage{esint}
\usepackage{dsfont}
\usepackage{relsize}
\usepackage{enumitem}
\usepackage{stmaryrd}
\usepackage{graphicx}

\usepackage{enumitem}
\setlist[enumerate,1]{label=(\arabic*), ref=(\arabic*)}

\usepackage{accents}

\usepackage{color,xcolor}

\usepackage{geometry}
\usepackage{hyperref}

\numberwithin{equation}{section}
\newtheorem{theorem}{Theorem}[section]
\newtheorem{corollary}[theorem]{Corollary}
\newtheorem{proposition}[theorem]{Proposition}
\newtheorem{lemma}[theorem]{Lemma}
\newtheorem{definition}[theorem]{Definition}

\newcommand{\R}{\mathbb{R}}

\newcommand{\N}{\mathbb{N}}

\newcommand{\p}{\partial}

\newcommand{\abs}[1]{\left\lvert#1\right\rvert}
\newcommand{\aabs}[1]{\left\lVert#1\right\rVert}
\newcommand{\eps}{\varepsilon}

\newcommand{\dis}[1]{\text{dis}#1}

\newcommand{\boundaryquasimetric}{\delta}
\newcommand{\quasiparameter}{\alpha}

\usepackage[framemethod=tikz]{mdframed}
\usepackage{url}

\newmdenv[backgroundcolor=blue!10,innerlinewidth=0.5pt,roundcorner=4pt,linecolor=aliceblue,innerleftmargin=6pt,innerrightmargin=6pt,innertopmargin=6pt,innerbottommargin=6pt]{mybox}

\newcommand{\dummy}{\,{\cdot}\,}

\title[Finsler travel time data]{Stable recovery of a simple irreversible Finsler geometry from travel time data}

\author[de Hoop]{Maarten V. de Hoop}
\address{M. V. de Hoop, Simons Chair in Computational and Applied Mathematics and Earth Science, Rice University, Houston, TX 77005, USA}
\email{mdehoop@rice.edu}

\author[Ilmavirta]{Joonas Ilmavirta}
\address{J. Ilmavirta, Department of Mathematics and Statistics, University of Jyv\"askyl\"a, Jyv\"askyl\"a, 40014,  Finland} \email{joonas.ilmavirta@jyu.fi}

\author[Kykkänen]{Antti Kykkänen}
\address{A. Kykkänen, Computational appliead mathematics and operations research
\\
Rice University
\\ 
TX 77005, USA}
\email{ak272@rice.edu }

\author[Saksala]{Teemu Saksala}
\address{T. Saksala, Department of Mathematics\\
North Carolina State University, Raleigh\\ 
NC 27695, USA}
\email{tssaksal@ncsu.edu}


\date{\today}

\begin{document}

\begin{abstract}
We show that a simple irreversible Finsler geometry can be recovered uniquely and Lipschitz-stably from its travel time data.
We introduce and use a version of Gromov--Hausdorff distance adapted to irreversible metric spaces.
In contrast to reversible (e.g. Riemannian) geometry, even the question of stability becomes ill-defined without simplicity.
\end{abstract}

\maketitle

\section{Introduction}

\subsection{Main results}

We show the uniqueness and stability of determining a smooth irreversible Finsler function from the travel time data corresponding to the arrival times of waves in a class of Finsler spaces. 
The propagation of waves in a medium can often be described in terms of a Hamiltonian flow. If the medium moves and the Hamiltonian is convex, the integral curves of the flow are (lifts of) geodesics of an irreversible Finsler function. Irreversibility means that a geodesics with reversed parametrization is not necessarily a geodesic or, equivalently, that $d_F(x,y) \neq d_F(y,x)$. The Finslerian distance function $d_F$ is therefore a \emph{quasimetric}, not a metric. 
A famous example of irreversible geometry arising from flowing medium is the Zermelo navigation problem: Given a vehicle moving in a medium with a current or wind, find a path that minimizes travel time between two points? In other words Zermelo's problem asks to find geodesics of a {Randers metric}.

We define the \emph{travel time map} of a smooth compact Finsler manifold $(M,F)$ with boundary $\partial M$ to be the map $r \colon M \to C(\partial M)$ taking a point $x$ to its \emph{travel time function} $r(x) = r_x$ defined by $r_x(z) = d_F(x,z)$ for all $z \in \partial M$. Since the Finsler function $F$ is not assumed to be symmetric, this is the distance form interior points to boundary points. The \emph{travel time data} $r(M)$ of $(M,F)$ is the range of the travel time map. We set out to study the inverse problem of stably recovering an irreversible Finsler function $F$ from the associated travel time data $r(M)$. For reasons that will be evident in a moment we restrict our attention to so called \emph{weakly simple} Finsler manifolds.

\begin{definition}
\label{def:simple_manifold}
Let $(M,F)$ be a smooth compact Finsler manifold with smooth boundary.
We say that $(M,F)$ is \emph{weakly simple} if (1) any two points are connected by a unique geodesic, (2) every maximal geodesic has two endpoints on $\partial M$ and minimizes the distance between these points, (3) every maximal geodesic meets the boundary at exactly two points.
\end{definition}

A simple Riemannian manifold as defined in~\cite[Definition 3.8.1]{paternain2023geometric} satisfies the definition above, but Definition~\ref{def:simple_manifold} is not strong enough to prove simplicity when the Finsler metric is Riemannian.
That is why we have chosen the term ``weak simplicity''. Definition~\ref{def:simple_manifold} implies that the manifold is non-trapping and has no interior conjugate points. Also, the
definition implies that the boundary is convex but not necessarily strictly convex~\cite[Corollary 1.2]{bartolo2011convex}.

Since $d_F$ is a quasi-metric on $M$ we can consider both the \emph{forward and backward metric balls}: 
\begin{equation}
\label{eqn:for-back-balls}
B_{r,+}(p)
:=\{x\in M: \: d_F(p,x)<r \}
\quad \text{and} \quad 
B_{r,-}(p)
:=\{x\in M: \: d_F(x,p)<r \}
\end{equation}
for $p \in M$ and $r > 0$.
These give rise the forward and backward metric topologies of $N$ which actually coincide with the manifold topology of $M$~\cite[Section 6.2C]{bao2012introduction}.
For completeness, we outline in Lemma~\ref{lma:quasi-reversibility} that the Finsler distance function is \emph{$\quasiparameter$-quasireversible} for some $\quasiparameter \in (0,1]$ in the sense that $d_F(x,y) \geq \quasiparameter d_F(y,x)$ for all $x,y \in M$.
Therefore the \emph{reversibility constant}
\begin{equation}
\label{eqn:reversibility-constant}
\quasiparameter_F
=
\inf_{x \neq y}
\frac{d_F(x,y)}{d_F(y,x)} \in (0,1]
\end{equation}
of $F$
is a well defined.
We have $\quasiparameter_F = 1$ if and only if $F$ is reversible. In general, we say that a quasimetric space $(Z,d)$ is \emph{$\quasiparameter$-quasireversible} if $d(p,q) \geq \quasiparameter d(q,p)$ for some uniform $\quasiparameter \in (0,1]$.

The main contribution of this paper is to prove  that the travel time data determines the irreversible Finsler manifold Lipschitz stably. To state the stability result we need to introduce the non-symmetric versions of Hausdorff and Gromov--Hausdorff distances for compact quasimetric spaces. We use this to measure the distance between two compact irreversible Finsler manifolds.

\begin{definition}
\label{def:hausdorff-gh-distance}
Let $\alpha \in (0,1]$.
For two compact subsets $A$ and $B$ of a $\quasiparameter$-quasireversible quasimetric space $Z$, their \emph{$\alpha$-Hausdorff distance} is given by
\begin{equation}
d_H^{Z,\alpha}(A,B)
=
\max
\left\{
\Gamma(A,B)
,
\alpha
\Gamma(B,A)
\right\}
\quad
\text{ where }
\Gamma(A,B)
=
\adjustlimits
\sup_{a \in A}
\inf_{b \in B}
d_Z(a,b).
\end{equation}

For two compact $\quasiparameter$-quasireversible quasimetric spaces $X$ and $Y$ their \emph{$\quasiparameter$-Gromov--Hausdorff} distance is defined by
\begin{equation}
d_{GH}^\quasiparameter(X,Y)
=
\inf_Z
d_H^{Z,\alpha}(X,Y)
\end{equation}
where the infimum is taken over all compact $\quasiparameter$-quasireversible quasimetric spaces $Z$ into which $X$ and $Y$ can be isometrically embedded.
\end{definition}

We show in Lemma~\ref{lma:hausdorff-quasimetric-space} that $d_H^Z$ is a $\quasiparameter$-quasireversible quasimetric in the space of compact sets of $Z$, while in Proposition~\ref{prop:gh-quasimetric} we prove that $d^\quasiparameter_{GH}$ defines an $\quasiparameter$-quasireversible quasimetric in the space of isometry classes of compact  $\quasiparameter$-quasi-reversible quasimetric spaces.

When $M$ is equipped with two weakly simple Finsler metrics we want to introduce a distance between the respective travel time data sets. With this in mind we first define a function $\boundaryquasimetric \colon C(\partial M) \times C(\partial M) 
\to
\R$
by
\begin{equation}
\boundaryquasimetric(f,h) 
=
\max_{\partial M} (f - h)
.
\end{equation}
For each $\quasiparameter\in (0,1]$ we define the $\quasiparameter$-quasisymmetrized version of this function:
\begin{equation}
\label{eqn:delta-alpha}
\boundaryquasimetric_\quasiparameter(f,h)
=
\max\{\boundaryquasimetric(f,h),\quasiparameter\boundaryquasimetric(h,f)\}
.
\end{equation}
This is a relaxed version of the supremum norm $\boundaryquasimetric_1$.

Since we are considering irreversible Finsler functions, we are forced to use a quasimetric on $C(\partial M)$.
If we were to have an isometry 
\\
$\phi\colon (M,F) \to (C(\p M), \aabs{\dummy}_{\infty})$, which exists for instance when $F$ is a weakly simple Riemannian metric \cite[Proposition 6]{ilmavirta2023three}, then $d_F$ would inherit symmetry from the norm. 
Therefore no symmetric norm on $C(\p M)$ is possible, and we cannot afford to fully symmetrize $\boundaryquasimetric_\quasiparameter$ by choosing $\quasiparameter=1$ unless the Finsler manifold $(M,F)$ is reversible to begin with.

The function $\boundaryquasimetric$ is neither symmetric nor necessarily positive, but $\boundaryquasimetric_\quasiparameter$ is a quasimetric (see Lemma~\ref{lem:delta_alpha_is_quasimetric}).
If $(M,F)$ is weakly simple, then $\boundaryquasimetric = \boundaryquasimetric_\quasiparameter$ on $r(M)$ for all small $\quasiparameter$ (see Lemma~\ref{lma:quasimetric-isometry}). 
If $(M,F)$ is not weakly simple we do not necessarily have $\boundaryquasimetric = \boundaryquasimetric_\quasiparameter$ for any positive $\quasiparameter$, and $\boundaryquasimetric$ can fail to be a quasimetric on $r(M)$.\footnote{Take a smooth radial wave speed $c(\abs{x})$ in the closed unit disk $\bar D$ so that $c(r)=1$ for $r>1/9$, $c(r)=1/100$ for $r<1/10$ and $c(r)$ takes values in between
in the annulus between these regions.
With the conformally Euclidean Riemannian metric $g=c^{-2}e$ it is a simple exercise to see that $\boundaryquasimetric((\frac19,0),(0,0))<0$.
Therefore $\boundaryquasimetric$ is not a quasimetric on $r_{g}(\bar D)$ and $\boundaryquasimetric = \boundaryquasimetric_\quasiparameter$ fails for every positive $\quasiparameter$.
A straight line through the origin is a geodesic between two antipodal points on $\partial D$ but not a minimizing one, so $(\bar D,g)$ is not weakly simple.}
Therefore without simplicity the travel time data set $r(M)\subset C(\partial M)$ appears to have no natural metric structure compatible with stability, which undermines our ability to even ask about stability of recovery.

In our main result we show that the travel time data determines the Finsler manifold Lipschitz stably in a class of weakly simple Finsler manifolds whose reversibility constants have a uniform lower bound.

\begin{theorem}
\label{thm:stability}
Let $\quasiparameter_0 \in (0,1]$ and let $M$ be a smooth compact manifold with smooth boundary. Then
\begin{equation}
d_{GH}^{\quasiparameter_0}((M,d_{F_1}),(M,d_{F_2}))
\leq
d_H^{\boundaryquasimetric_{\quasiparameter_0},\quasiparameter_0}(r_1(M),r_2(M))
\end{equation}
for all weakly simple Finsler metrics on $M$ whose reversibility constants satisfy $\quasiparameter_{F_1},\quasiparameter_{F_2} \geq \quasiparameter_0$. 
\end{theorem}

We record two corollaries to Theorem~\ref{thm:stability}. The first one is a symmetrized stability result.

\begin{corollary}
\label{cor:stability-symm}
Let $\quasiparameter_0 \in (0,1]$ and let $M$ be a smooth compact manifold with smooth boundary. Then
\begin{equation}
\begin{split}
\max
\{
d_{GH}^{\quasiparameter_0}&((M,d_{F_1}),(M,d_{F_2})),
d_{GH}^{\quasiparameter_0}((M,d_{F_2}),(M,d_{F_1}))
\}
\\
&\leq
\frac{1}{\quasiparameter_0}
\min
\left\{
d_H^{\boundaryquasimetric_{\quasiparameter_0},\alpha_0}(r_1(M),r_2(M)),
d_H^{\boundaryquasimetric_{\quasiparameter_0},\alpha_0}(r_2(M),r_1(M))
\right\}
\end{split}
\end{equation}
for all weakly simple Finsler metrics on $M$ whose reversibility constants satisfy $\quasiparameter_{F_1},\quasiparameter_{F_2} \geq \quasiparameter_0$.
\end{corollary}

In our second corollary to Theorem~\ref{thm:stability} we verify that the travel time data determines a weakly simple irreversible Finsler function up to a boundary fixing Finsler isometry.
Corollary~\ref{cor:uniqueness} below is contained in~\cite[Theorem 1.3]{dehoop2020determination}, but we reprove it here using a version of the Myers--Steenrod theorem; the techniques of~\cite{dehoop2020determination} were very different.

\begin{corollary}
\label{cor:uniqueness}
Let $F_1$ and $F_2$ be two weakly simple Finsler metrics on $M$ such that $r_1(M) = r_2(M)$. Then there is a Finsler isometry $\phi \colon (M,F_1) \to (M,F_2)$ that is an identity on $\p M$.
\end{corollary}

Lastly, we list the result that the metric structure we use to compare the similarity of the irreversible compact Finsler manifolds is indeed well behaved.

\begin{proposition}
\label{prop:gh-quasimetric}
For any $\quasiparameter \in (0,1]$ Gromov--Hausdorff distance $d^\quasiparameter_{GH}$, as in Definition~\ref{def:hausdorff-gh-distance}, defines a finite $\quasiparameter$-quasireversible quasimetric in the space of isometry classes of compact $\quasiparameter$-quasireversible quasimetric spaces. That is,
\begin{enumerate}
    \item \label{item:non-neg} $d^\quasiparameter_{GH}$ is non-negative,
    \item \label{item:triangle-ineq} $d^\quasiparameter_{GH}$ satisfies the triangle inequality,
    \item \label{item:vanishing-gh} $d^\quasiparameter_{GH}(X,Y) = 0$ if and only if there is a bijective isometry between the $\quasiparameter$-quasireversible quasimetric spaces $X$ and $Y$, 
    \item \label{item:quasireversible} $d^{\quasiparameter}_{GH}(X,Y) \geq \quasiparameter d^{\quasiparameter}_{GH}(Y,X)$ for any compact $\quasiparameter$-quasireversible quasimetric spaces $X$ and $Y$.
\end{enumerate}
\end{proposition}

\subsection{Relation to other inverse problems}
As any inverse problems, also the geometric ones come with questions of uniqueness and stability.
Unique determination of a Riemannian manifold from boundary distance data was shown in~\cite{Katchalov2001}, and Lipschitz stability was shown in~\cite{ilmavirta2023three} under the assumption of simplicity.
Stability estimates exist also in more general geometry~\cite{katsuda2007stability}, but it is not Lipschitz.

Unique determination of a Finsler manifold from travel time data was shown in~\cite{dehoop2020determination}, but it came with natural obstruction, unlike the Riemannian version.
Namely, unique determination of the Finsler function $F\colon TM\to\R$ only holds on a subset of the tangent bundle $TM$ and outside this set one can deform $F$ to get a new Finsler function $\tilde F$ whose travel time data cocides with that of $F$ but $\tilde F$ and $F$ are not isometric.
For full uniqueness, one needs to make assumptions like fiberwise real analyticity\footnote{Riemannian manifolds are fiberwise real analytic Finsler manifolds.} (so that this subset is enough for full determination by analytic continuation on each fiber) or simplicity (so that this subset is the whole bundle).
The present paper does to~\cite{dehoop2020determination} what~\cite{ilmavirta2023three} did to~\cite{Katchalov2001}: we show Lipschitz stability under a simplicity assumption.

Stable reconstruction of a reversible Finsler manifold was studied in~\cite{flie}.
Reversibility brings about substantial technical ease, and for irreversible Finsler manifolds the whole question of stability becomes ill-behaved as
described above.

In order to treat irreversible geometry, we introduce a notion of Gromov--Haus-dorff distance adapted to such geometries.
This is analogous to the introduction of a labeled Gromov--Hausdorff distance in~\cite{de2021stable} in order to estimate Riemannian manifolds not just abstractly but in relation to the known boundary.

Finsler manifolds are less rigid than Riemannian ones, as mentioned above in relation to~\cite{dehoop2020determination}.
Similarly, boundary rigidity (unique determination of a manifold from all pairwise distances of boundary points) is known on some Riemannian manifolds~\cite{
Besson1995Entropies,
croke1990rigidity,
croke1991rigidity,
gromov1983filling,
mukhometov1981problem,
mukhometov1978problem,
Otal1990,
pestov2005two,
stefanov2016boundary,
stefanov2017local%
}
(including stability~\cite{stefanov2005boundary}) but no such rigidity is possible on a Finsler manifold due to Ivanov's counterexample~\cite{Ivanov:volume-monotonicity}.
The counterexample is based on the freedom to modify the Finsler structure on the whole tangent bundle without any regard for individual fibers. 
The recent work~\cite{ES26:hamilton-finsler} gives a rigidity result for Finsler geometries, but the gauge freedom is a certain symplectomorphism of $T^*M$, not just diffeomorphism of $M$.

Several mechanisms can make a Finsler manifold more rigid.
One option is to look at a narrower class of Finsler manifolds, as in the Randers boundary rigidity of~\cite{monkkonen2020boundary}.
Another option is to use a kind of data that in some way ties together tangent vectors belonging to the same fiber.
A Finsler manifold was reconstructed from its broken scattering relation in~\cite{de2020foliated}, and a breaking point is where the velocity of a geodesic suddenly jumps to another vector of the same tangent space.
In the present paper a similar effect is caused by, in effect, studying all geodesics emanating from a single point and repeating it for all points.
Yet another alternative is to reconstruct the Finsler geometry only along a reference geodesic as in~\cite{Finsler_Dix}.

\subsection*{Acknowledgments}
MVdH was supported by the Simons Foundation under the MATH + X program, the National Science Foundation under grant DMS-2108175, and the corporate members of the GeoMathematical Imaging Group at Rice University.
AK was supported by the corporate members of the GeoMathematical Imaging Group at Rice University.
JI was supported by the Research Council of Finland (Flagship of Advanced Mathematics for Sensing Imaging and Modelling grant 359208; Centre of Excellence of Inverse Modelling and Imaging grant 353092; and other grants 351665, 351656, 358047, 360434) and a Väisälä project grant by the Finnish Academy of Science and Letters.
TS was supported by the National Science Foundation (DMS-2510272) and the  Simons Foundation Travel Support for Mathematicians (MPS-TSM-00013291).

\section{Proofs of the main theorems}
\label{sec:proofs-of-main-theorems}

\begin{lemma}
\label{lma:quasimetric-isometry}
If $(M,F)$ is a weakly simple Finsler manifold then for all $0<\quasiparameter \leq \quasiparameter_F$ we have the equation $\boundaryquasimetric = \boundaryquasimetric_\quasiparameter$ on $r(M)$. In particular, the travel time map $r \colon M \to C(\partial M)$ is an isometric embedding in the sense that $\boundaryquasimetric(r(x),r(y)) = d(x,y)$ for all $x,y \in M$.
\end{lemma}

\begin{proof}
Suppose that $x \neq y$. By the triangle inequality all $x,y \in M$ and $z \in \partial M$ satisfy
\begin{equation}
\label{eqn:triangle-ineq-1}
r_x(z) - r_y(z) = d(x,z) - d(y,z) \leq [d(x,y) + d(y,z)] - d(y,z) \leq d(x,y)
\end{equation}
and
\begin{equation}
\label{eqn:triangle-ineq-2}
r_x(z) - r_y(z) = d(x,z) - d(y,z) \geq d(x,z) - [d(y,x) + d(x,z)] \geq -d(y,x).
\end{equation}
We will next show that these bounds are optimal.

Let $\gamma_{xy}$ be the unique geodesic going from $x$ to $y$, extended maximally into $\gamma_{x_\mathrm{in}y_\mathrm{out}}$ so that the endpoints $x_\mathrm{in}$ and $y_\mathrm{out}$ are on the boundary.
Because the geodesic $\gamma_{x_\mathrm{in}y_\mathrm{out}}$ is a distance minimizer by weak simplicity and the four points lay on it in the order $x_\mathrm{in}$, $x$, $y$, $y_\mathrm{out}$, evaluating $r_x$ and $r_y$ at these endpoints gives
\begin{equation}
\begin{split}
r_x(x_\mathrm{in}) - r_y(x_\mathrm{in}) &= -d(y,x)
\quad\text{and}\\
r_x(y_\mathrm{out}) - r_y(y_\mathrm{out}) &= d(x,y).
\end{split}
\end{equation}
Combining these with~\eqref{eqn:triangle-ineq-1} and~\eqref{eqn:triangle-ineq-2} gives
\begin{equation}
\min_{\partial M}(r_x - r_y) = -d(y,x)
\quad\text{and}\quad
\max_{\partial M}(r_x - r_y) = d(x,y)
\end{equation}
as claimed.

Due to Lemma~\ref{lma:quasi-reversibility} the quasireversibility constant $\quasiparameter_F$ exists. Let $\quasiparameter\in (0,\quasiparameter_F]$ and $x,y \in M$. 
Then by the first part of the proof we have that
\begin{equation}
\boundaryquasimetric(r_x,r_y)=d(x,y)
\geq 
\quasiparameter d(y,x)
=
\quasiparameter \boundaryquasimetric(r_y,r_x),
\end{equation}
and so 
$
\boundaryquasimetric(r_x,r_y)
=
\boundaryquasimetric_\quasiparameter (r_x,r_y)
$.
\end{proof}

\begin{lemma}[Myers--Steenrod]
\label{lma:myers-steenrod}
Let $(M_1,F_1)$ and $(M_2,F_2)$ be two smooth Finsler manifolds with boundary. If $f \colon M_1 \to M_2$ is an isometry between quasimetric spaces $(M_1,d_{F_1})$ and $(M_2,d_{F_2})$
then $f$ is a smooth map and $F_1 = f^*F_2$.
\end{lemma}

\begin{proof}
This proof is a variant of the similar argument given in~\cite[Proposition 24]{flie} for the reversible case. We present a condensed version to showcase that reversibility was in no way crucial.

By invariance of domain $f$ maps the interior to the interior and the boundary to the boundary. The interior map $\psi_{\mathrm{int}}=f|_{\mathrm{int}(M_1)} \colon \mathrm{int}(M_1) \to \mathrm{int}(M_2)$ is clearly an isometry.
The intrinsic distance on $\partial M_i$ induced by $F_i$ can be computed from $F_i$ as the induced path metric, and therefore the boundary restriction $\psi_{\partial}=f_{\partial M_1} \colon \partial M_1 \to \partial M_2$ of $f$ is an isometry.

By the Myers--Steenrod theorem for Finsler manifolds without boundary~\cite[Theorem A]{matveev2017myers}, both $\psi_{\mathrm{int}}$ and $\psi_\partial$ are smooth isometries and 
\begin{equation}
\label{eqn:interior-pullback}
F_1 = \psi_{\mathrm{int}}^*F_2
\quad\text{in}\quad
\mathrm{int}(M_1).
\end{equation}
To prove that $f$ is smooth, we need to ensure that the interior and boundary diffeomorphisms are glued together smoothly. This follows from the boundary normal coordinates being smooth near the boundary, which holds without reversibility~\cite[Lemma 3.3]{dehoop2020determination}. Therefore $f$ is smooth, and the same argument applies to $f^{-1}$ as well.

Now that $f$ is smooth, the interior pullback property~\eqref{eqn:interior-pullback} implies $F_1 = f^*F$ on all of $M_1$.
\end{proof}

\begin{proof}[Proof of Theorem~\ref{thm:stability}]
Due to Lemma~\ref{lma:quasimetric-isometry}, the travel time maps $r_i$ are isometric embeddings of $(M,F_i)$ into the quasimetric space $(C(\partial M),\boundaryquasimetric_{\quasiparameter_{F_i}})$ where $\quasiparameter_{F_i}$ is the reversibility constant of $F_i$, and $\boundaryquasimetric = \boundaryquasimetric_{\quasiparameter_{F_i}}$ on $r_i(M)$ for $i = 1,2$. These statements remain true for all parameter values $\quasiparameter \leq \quasiparameter_{F_i}$, so in particular, they are true for $\quasiparameter_0 \leq \quasiparameter_{F_1},\quasiparameter_{F_2}$. Hence, both travel time maps $r_i$ are isometric embdeddings of $(M,F_i)$ into a common quasimetric space $(C(\partial M),\boundaryquasimetric_{\quasiparameter_0})$. Since in the definition of the Gromov--Hausdorff distance we can choose $Z = (C(\partial M),\boundaryquasimetric_{\quasiparameter_0})$  we arrive at the inequality
\begin{equation*}
d_{GH}^{\quasiparameter_0}((M,d_{F_1}),(M,d_{F_2}))
\leq
d^{\boundaryquasimetric_{\quasiparameter_0},\alpha_0}_H(r_1(M),r_2(M)).
\qedhere
\end{equation*}
\end{proof}

\begin{proof}[Proof of Corollary~\ref{cor:stability-symm}]
Due to Lemma~\ref{lma:hausdorff-quasisymmetric} it holds that
\begin{equation}
d^{\delta_{\quasiparameter_0},\alpha_0}_H(r_1(M),r_2(M))
\leq
\frac{1}{\quasiparameter_0}
\min
\left\{
d^{\delta_{\quasiparameter_0},\alpha_0}_H(r_1(M),r_2(M)),
d^{\delta_{\quasiparameter_0},\alpha_0}_H(r_2(M),r_1(M))
\right\}
\end{equation}
and the same holds when the left hand-side is replaced by $d^{\delta_{\quasiparameter_0},\alpha_0}_H(r_2(M),r_1(M))$.
Thus the claim follows from Theorem~\ref{thm:stability}.
\end{proof}

\begin{proof}[Proof of Corollary~\ref{cor:uniqueness}]
Since $r_1(M) = r_2(M)$, it follows from Theorem~\ref{thm:stability} that $d^{\quasiparameter_0}_{GH}((M,d_{F_1}),(M,d_{F_2})) = 0$ for suitable $\quasiparameter_0$. Thus by Proposition~\ref{prop:gh-quasimetric} there is a map $\phi \colon M \to M$ so that $d_{F_1}(x,x') = d_{F_2}(\phi(x),\phi(x'))$. This map $\phi$ is smooth and $F_1 = \phi^*F_2$ by Lemma~\ref{lma:myers-steenrod}. In other words, $\phi$ is a Finsler isometry between $(M,F_1)$ and $(M,F_2)$.

To see that $\phi$ fixes all boundary points, let $z \in \partial M$ and $z' \coloneqq \phi(z)$.
Then $r_1(z) = r_2(z')$, so evaluating the first map at $z$ gives $0 = d_{F_2}(z',z)$ and so $z' = z$.
\end{proof}

\section{Gromov--Hausdorff distance of quasireversible quasimetric spaces}
\label{sec:rev-constant-quasimetrics}

This section establishes the basic properties of the Hausdorff and Gromov--Hausdorff distance for $\quasiparameter$-quasireversible quasimetric spaces as defined in Definition~\ref{def:hausdorff-gh-distance}. Particularly, we prove in here Propositon~\ref{prop:gh-quasimetric}. Before doing this we prove that the space of compact subsets of a $\quasiparameter$-quasireversible quasimetric space forms a $\quasiparameter$-quasireversible quasimetric space when equipped with the Hausdorff metric (see Lemma~\ref{lma:hausdorff-quasimetric-space}). 
In Lemma~\ref{lma:hausdorff-quasisymmetric} we also record an elementary estimate used in the proof of Theorem~\ref{thm:stability}.

For $\quasiparameter$-quasireversible quasimetric space we define the forward and backward topologies using the respective forward and backward balls as in~\eqref{eqn:for-back-balls}.
These topologies as well as the symmetrized\footnote{The topology induced by the distance function $d^s(p,q) \coloneqq \max\{d(p,q),d(q,p)\}$} 
topology are all bi-Lipschitz equivalent.
Hence, the compact subsets in these topologies are the same.

\begin{lemma}
\label{lma:hausdorff-quasisymmetric}
Let $\quasiparameter \in (0,1]$ and let $Z$ be a compact $\quasiparameter$-quasireversible quasimetric space. Then for all compact subsets $A,B \subset Z$ we have
\begin{enumerate}[itemsep=1.2ex]
    \item \label{item:hausdorff-quasireversible} $d^{Z,\alpha}_H(A,B) \geq \quasiparameter d^{Z,\alpha}_H(B,A)$
    \item \label{item:hausdorff-to-symm} $\quasiparameter d^{Z,\alpha}_H(A,B) \leq \min\{d^{Z,\alpha}_H(A,B),d^{Z,\alpha}_H(B,A)\}$.
\end{enumerate}
\end{lemma}

\begin{proof}
Item~\ref{item:hausdorff-quasireversible} follows directly from definitions: Let $A,B \subset Z$ be compact. Then
\begin{equation}
\begin{split}
\alpha d^{Z,\alpha}_H(B,A)
&=
\max
\{
\alpha\Gamma(B,A),
\alpha^2\Gamma(A,B)
\}
\\
&\leq
\max
\{
\alpha\Gamma(B,A),
\Gamma(A,B)
\}
\\
&=
d^{Z,\alpha}_H(A,B)
\end{split}
\end{equation}
as claimed.

Item~\ref{item:hausdorff-to-symm} follows from item~\ref{item:hausdorff-quasireversible}: Indeed, since $\quasiparameter \leq 1$, we have
$
\quasiparameter d^{Z,\alpha}_H(A,B) \leq d^{Z,\alpha}_H(A,B)
$
and due to item~\ref{item:hausdorff-quasireversible}, we have
$
\quasiparameter d^{Z,\alpha}_H(A,B) \leq d^{Z,\alpha}_H(B,A).
$
\end{proof}

\begin{lemma}
\label{lma:hausdorff-quasimetric-space}
Let $\quasiparameter\in (0,1]$ and $(Z,d)$ be a compact $\quasiparameter$-quasireversible quasimetric space. Then the Hausdorff distance $d_H^{Z}$ is an $\quasiparameter$-quasireversible quasimetric in $\mathcal{C}_Z \coloneqq \{A \subset Z:A \text{ compact}\}$.
\end{lemma}

\begin{proof}
By Lemma~\ref{lma:hausdorff-quasisymmetric} $d^Z_H$ is $\quasiparameter$-quasireversible. It remains to show that $d^Z_H$ is a quasimetric.

We observe that for $A,B \in \mathcal{C}_Z$ we have $d^{Z,\alpha}_H(A,B) = 0$ if and only if $A = B$.
Clearly $d^{Z,\alpha}_H(A,A) = 0$.
If there is $x\in A\setminus B$, then by compactness $d_Z(x,B)>0$ and so $d^{Z,\alpha}_H(A,B) > 0$.
Similarly any $x\in B\setminus A$ leads to $d^{Z,\alpha}_H(A,B) > 0$.

Let us then show that $d^Z_H$ satisfies the triangle inequality. Let $A,B,C \in \mathcal{C}_Z$, 
$a \in A$ and $b \in B$. Since $d_Z$ is a quasimetric, for all $c \in C$, we have
\begin{equation}
\inf_{b \in B}
d(a,b)
\leq
d(a,b)
\leq
d(a,c)
+
d(c,b).
\end{equation}
Taking infimum over $b \in B$ in the right-hand side, we get
\begin{equation}
\begin{split}
\inf_{b \in B}
d(a,b)
&\leq
d(a,c)
+
\inf_{b \in B}
d(c,b)
\leq
d(a,c)
+
\adjustlimits
\sup_{c \in C}
\inf_{b \in B}
d(c,b).
\end{split}
\end{equation}
Since $c \in C$ is arbitrarily chosen we have that
\begin{equation}
\inf_{b \in B}
d(a,b)
\leq
\inf_{c \in C}
d(a,c)
+
\adjustlimits
\sup_{c \in C}
\inf_{b \in B}
d(c,b).
\end{equation}
Therefore,
$
\Gamma(A,B)
\leq
\Gamma(A,C)
+
\Gamma(C,B)
$
and finally,
\begin{equation}
\begin{split}
d^{Z,\alpha}_H(A,B)
&=
\max\{\Gamma(A,B),\alpha\Gamma(B,A)\}
\\
&\leq
\max\{\Gamma(A,C),\alpha\Gamma(C,A)\}
+
\max\{\Gamma(C,B),\alpha\Gamma(B,C)\}
\\
&\leq
d^{Z,\alpha}_H(A,C)
+
d^{Z,\alpha}_H(C,B).
\end{split}
\end{equation}
The proof is finished.
\end{proof}

We are ready to prove Proposition~\ref{prop:gh-quasimetric}.

\begin{proof}[Proof of Proposition~\ref{prop:gh-quasimetric}]
\ref{item:non-neg}: The non-negativity of $d^{\quasiparameter}_{GH}$ is clear from the definitions. 

\ref{item:triangle-ineq}: For the triangle inequality we follow the steps for the proof of \cite[Prop. 7.3.16]{burago2001course} concerning the triangle inequality of the regular Gromov-Hausdorff distance. 
We consider compact $\quasiparameter$-quasireversible quasimetric spaces $X$, $Y$ and $W$, and choose $\quasiparameter$-quasireversible quasimetric spaces $Z_1$ and $Z_2$ so that $X,W \subset Z_1$ and $W,Y \subset Z_2$ as isometric embeddings.
We define $Z = Z_1 \cup Z_2$ with a function $d_Z$ that agrees with $d_{Z_i}$ in $Z_i\times Z_i$, 
\begin{equation}
d_Z(p,q)
=
\inf_{w \in W}
\{
d_{Z_1}(p,w)
+
d_{Z_2}(w,q)
\},
\text{ for } p\in Z_1, q\in Z_2
\end{equation}
and
\begin{equation}
d_Z(q,p)
=
\inf_{w \in W}
\{
d_{Z_2}(q,w)
+
d_{Z_1}(w,p)
\},
\text{ for } q\in Z_2, p\in Z_1.
\end{equation}
Since $Z_1$ and $Z_2$ are $\quasiparameter$-quasireversible quasimetric spaces and $d_{Z_1}=d_{Z_2}$ in $W\times W$ it follows from a direct, but somewhat lengthy computation, that $d_Z$ is a well defined $\quasiparameter$-quasireversible quasimetric in $Z$. Thus, the  triangle inequality for $d_{GH}^\quasiparameter$ follows from Lemma~\ref{lma:hausdorff-quasimetric-space}.

\ref{item:vanishing-gh}:
If $X$ and $Y$ are isometric compact $\alpha$-quasireversible quasimetric spaces, it straightforward to see that $d^\alpha_{GH}(X,Y) = 0$ by embedding, say, $X$ into $Y$ isometrically and using Lemma~\ref{lma:hausdorff-quasimetric-space}.
Let $d^\quasiparameter_{GH}(X,Y) = 0$.
In order to show that $X$ and $Y$ are isometric we follow the steps of \cite[Theorem 7.3.30]{burago2001course}.
For each $n \in \N$ there is a compact $\quasiparameter$-quasireversible quasimetric space $Z_n$ and isometric embeddings $i_n \colon X \to Z_n$ and $j_n \colon Y \to Z_n$ such that
\begin{equation}
\eps_n
\coloneqq
d_H^{Z_n}(i_n(X),j_n(Y)) \to 0,
\quad 
\text{as } n \to \infty.
\end{equation}
Thus, by the definition of $\alpha$-Hausdorff distance (Definition~\ref{def:hausdorff-gh-distance}), we have
\begin{equation}
\begin{split}
\adjustlimits
\sup_{x \in X}
\inf_{y \in Y}
d_{Z_n}(i_n(x),j_n(y))
\leq
\eps_n
\quad
\text{and}
\quad
\adjustlimits
\sup_{y \in Y}
\inf_{x \in X}
d_{Z_n}(j_n(y),i_n(x))
\leq
\frac{\eps_n}{\quasiparameter}.
\end{split}
\end{equation}
Let us choose maps $f_n \colon X \to Y$ and $g_n \colon Y \to X$ such that
\begin{equation}
\label{eq:estimate_for_f_n_and_g_n}
\begin{split}
d_{Z_n}(i_n(x),j_n(f_n(x)))
&\leq
\eps_n + \frac{1}{n}
\eqqcolon
\eta_n
\quad
\text{and}
\\
d_{Z_n}(j_n(y),i_n(g_n(y)))
&\leq
\frac{\eps_n}{\quasiparameter} + \frac{1}{n}
\eqqcolon
\theta_n.
\end{split}
\end{equation}
Such maps can be chosen by the definition of supremum and infimum.

Using the triangle inequality we see that
\begin{equation}
\begin{split}
d_X(x,g_n(f_n(x))), 
\: d_Y(y,f_n(g_n(y)))
&\leq
\eta_n + \theta_n,
\end{split}
\end{equation}
for all $x \in X$ and $y \in Y$. Therefore,
\begin{equation}
\label{eqn:composition-convergence-1}
\begin{split}
\sup_{x \in X}
d_X(x,g_n(f_n(x)))
, \: 
\sup_{y \in Y}
d_Y(y,f_n(g_n(y)))
\to 0,
\quad \text{as } n \to \infty. 
\end{split}
\end{equation}
Due to $\quasiparameter$-quasireversibility of $X$ and $Y$ it also follows that
\begin{equation}
\label{eqn:composition-convergence-2}
\begin{split}
\beta_n
\coloneqq
\sup_{x \in X}
d_X(g_n(f_n(x)),x)
&\to 0
\quad\text{and}
\\
\gamma_n
\coloneqq
\sup_{y \in Y}
d_Y(f_n(g_n(y)),y)
&\to 0,
\quad \text{as } n \to \infty. 
\end{split}
\end{equation}

We show next that the \emph{distortion} of the map $f_n$ satisfies
\begin{equation}
\label{eqn:uniform-convergence-1}
\dis(f_n)\coloneqq
\sup_{x,x' \in X}
\abs{d_Y(f_n(x),f_n(x')) - d_X(x,x')}
\to
0.
\end{equation}

Since, $d_X(x,x')=d_{Z_n}(i_n(x),i_n(x'))$  for all $x,x' \in X$, we have due to triangle inequality, \eqref{eq:estimate_for_f_n_and_g_n} and \eqref{eqn:composition-convergence-1} that
\begin{equation}
\abs{d_Y(f_n(x),f_n(x'))-d_X(x,x')}
\leq
\eta_n + \theta_n + \beta_n.
\end{equation}
The estimate \eqref{eqn:uniform-convergence-1} follows.
Similarly, one can show that
\begin{equation}
\label{eqn:uniform-convergence-2}
\dis(g_n)\coloneqq
\sup_{y,y' \in Y}
\abs{d_X(g_n(y),g_n(y')) - d_Y(y,y')}
\to 0.
\end{equation}

Next, we show that, after passing to a subsequence, there are distance preserving functions $\Phi \colon X \to Y$ and $\Psi \colon Y \to X$ so that $f_n \to \Phi$ and $g_n \to \Psi$ uniformly.
For the purposes of showing the above, let us consider the symmetrized distance functions
\begin{equation}
\begin{split}
d^s_X(x,x')
&=
\max
\{
d_X(x,x'),
d_X(x',x)
\},
\quad\text{and}
\\
d^s_Y(y,y')
&=
\max
\{
d_Y(y,y'),
d_Y(y',y)
\}.
\end{split}
\end{equation}
The quasimetric functions $d_X$ and $d_Y$ are continuous with respect to $d^s_X$ and $d^s_Y$, respectively. It holds that
\begin{equation}
\label{eqn:distortion-estimate}
\abs{
d_Y(f_n(x),f_n(x'))
-
d_X(x,x')
}
\leq
\dis(f_n)
\quad\text{for all } x,x' \in X.
\end{equation}
Hence
\begin{equation}
\label{eqn:asymptotic-continuity}
d^s_Y(f_n(x),f_n(x'))
\leq
d_X^s(x,x')
+
\dis(f_n)
\quad\text{for all } x,x' \in X.
\end{equation}
Choose, by the separability of a compact space $(X,d^s_X)$, a dense subset $S = \{ x_k \}_{k \in \N} \subset X$. Since $(Y,d^s_Y)$ is compact we can use the diagonal argument to extract a subsequence, still denoted by $(f_n)_{n = 1}^\infty$, such that $(f_n(x_k))_{n = 1}^\infty$ converges in $Y$ for all $k \in \N$. Define $\tilde\Phi \colon S \to Y$ by setting
\begin{equation}
\label{eqn:definition-on-s}
\tilde\Phi(x_k)
=
\lim_{n \to \infty}f_n(x_k)
\end{equation}
for all $k \in \N$. It follows from~\eqref{eqn:distortion-estimate} and continuity of $d_Y$ with respect to $d^s_Y$ that
\begin{equation}
d_Y(\tilde\Phi(x_k),\tilde\Phi(x_l))
=
d_X(x_k,x_l)
\end{equation}
for all $k,l \in \N$. Furthermore, it holds that
\begin{equation}
d_Y^s(\tilde\Phi(x_k),\tilde\Phi(x_l))
=
d_X^s(x_k,x_l)
\end{equation}
for all $k,l \in \N$. Hence,  $\tilde\Phi$ is uniformly continuous. Thus, by completeness of $(X,d^s_X)$ there is a unique continuous extension $\Phi \colon X \to Y$. Using  density of $S$ and continuity of the quasimetrics with respect to the symmetrized metrics one can show that 
\begin{equation}
\label{eq:Phi_is_isometry}
d_Y(\Phi(x),\Phi(x'))
=
d_X(x,x')
\end{equation}
for all $x,x' \in X$.

We show that $f_n \to \Phi$ uniformly with respect to the symmetrized metrics. Let $r > 0$ be arbitrary. Using compactness of $(X,d_X^s)$ again, there are $x_1,\dots,x_N \in S$ such that for all $x \in X$
\begin{equation}
\min_{1 \leq i \leq N} d^s_X(x,x_i) \leq r.
\end{equation}
Let $x \in X$ and choose $x_i$ as above so that $d^s_X(x,x_i) \leq r$. Then triangle inequality,  \eqref{eqn:asymptotic-continuity}, and  \eqref{eq:Phi_is_isometry} give
\begin{equation}
\begin{split}
d^s_Y(f_n(x),\Phi(x))
&\leq
2r
+
\dis(f_n)
+
\max_{1 \leq i \leq N}
d^s_Y(f_n(x_i),\Phi(x_i)).
\end{split}
\end{equation}
By taking $n$ large enough we get from~\eqref{eqn:uniform-convergence-1} and~\eqref{eqn:definition-on-s} that the right hand-side is less than $4r$. Thus $f_n \to \Phi$ uniformly. By an analogous argument we show that there is a distance preserving map $\Psi \colon Y \to X$ so that $g_n \to \Psi$ uniformly.

Finally, we show that $\Psi$ and $\Phi$ are inverses to each other. We take $x \in X$ and use~\eqref{eqn:asymptotic-continuity}, to see that
\begin{equation}
\begin{split}
d^s_X(g_n(f_n(x),\Psi(\Phi(x)))
&\leq
d^s_Y(f_n(x),\Phi(x))
+
\dis(g_n)
+
\sup_{y \in Y}d^s_Y(g_n(y),\Psi(y)).
\end{split}
\end{equation}
Therefore, $g_n \circ f_n \to \Psi \circ \Phi$ uniformly. On the other hand, by~\eqref{eqn:composition-convergence-1} and~\eqref{eqn:composition-convergence-2} we see that $g_n \circ f_n \to id_X$ uniformly. Hence $\Psi \circ \Phi = id_X$. Similarly, one shows that $\Phi \circ \Psi = id_Y$ proving that $\Phi \colon X \to Y$ is a bijective isometry.

\ref{item:quasireversible}: This claim follows from Definition~\ref{def:hausdorff-gh-distance} and Lemma~\ref{lma:hausdorff-quasisymmetric} after embedding $X$ and $Y$ isometrically in to a common $\quasiparameter$-quasireversible quasimetric space $Z$.
\end{proof}

\appendix

\section{Quasireversiblity of Finsler manifolds}

This appendix serves a two-fold purpose; we prove that the function $\delta_{\quasiparameter}$ defined in~\eqref{eqn:delta-alpha} is a quasimetric in $C(\partial M)$ and we prove that any compact Finsler manifold is $\quasiparameter$-quasireversible for some $\quasiparameter \in (0,1]$.

\begin{lemma}
\label{lem:delta_alpha_is_quasimetric}
Let $M$ be a smooth compact manifold with boundary. For each $\quasiparameter \in (0,1]$ the function $\boundaryquasimetric_\quasiparameter$ is a $\quasiparameter$-quasireversible quasimetric in $C(\partial M)$.
\end{lemma}

\begin{proof}
Let $f,h,g\in C(\p M)$.
The quasireversibility estimate
\begin{equation}
\label{eq:reversibility_cost}    
\boundaryquasimetric_{\quasiparameter}(f,g)
\geq
\quasiparameter\boundaryquasimetric_{\quasiparameter}(g,f),
\end{equation}
follows easily from the definitions.

For non-negativity, it is enough to notice that $\boundaryquasimetric(f,g)\leq0$ implies $\boundaryquasimetric(g,f)\geq0$.
Also, $\boundaryquasimetric_\quasiparameter(f,g)=0$ implies that both the minimum and maximum of $f-g$ is zero, and so $f=g$.

The triangle inequality can be split in two cases.
If
$
\boundaryquasimetric(f,h)\geq \quasiparameter  \boundaryquasimetric(h,f)
$,
then   
    \begin{equation}
   \boundaryquasimetric_\quasiparameter(f,h)
   =
   \boundaryquasimetric(f,h)
   \leq 
   \boundaryquasimetric(f,g)+\boundaryquasimetric(g,h)
   \leq 
    \boundaryquasimetric_\quasiparameter(f,g)+\boundaryquasimetric_\quasiparameter(g,h).
   \end{equation}
   If instead
   $
   \quasiparameter  \boundaryquasimetric(h,f)\geq \boundaryquasimetric(f,h)
   $,
   then
   \begin{equation*}
   \boundaryquasimetric_\quasiparameter(f,h)
   =
   \quasiparameter\boundaryquasimetric(h,f)
   \leq
   \quasiparameter(\boundaryquasimetric_\quasiparameter(h,g)+\boundaryquasimetric_\quasiparameter(g,f))
   \leq \boundaryquasimetric_\quasiparameter(f,g)+\boundaryquasimetric_\quasiparameter(g,h).
   \qedhere
   \end{equation*}
\end{proof}

Lemma~\ref{lma:quasi-reversibility} shows that the reversibility constant defined in~\eqref{eqn:reversibility-constant} is positive for any compact Finsler manifold. It is clear by definition that $\quasiparameter_F \leq 1$, and that $\quasiparameter_F = 1$ if and only if the metric is reversible.

\begin{lemma}[Quasireversibility]
\label{lma:quasi-reversibility}
Let $(M,F)$ be a compact Finsler manifold. Then there is $\quasiparameter \in (0,1]$ so that the Finslerian distance function satisfies $d_F(x,y) \geq \quasiparameter d_F(y,x)$.
\end{lemma}

\begin{proof}
By compactness the zero-homogeneous function $f \colon TM \setminus \{0\} \to \R$ defined by
$f(x,v)
=
\frac{F(x,v)}{F(x,-v)}$
attains its minimum and maximum, and the claim follows.
For more detail, see the proof of a local version in~\cite[Lemma 6.2.1]{bao2012introduction}.
\end{proof}

\bibliographystyle{abbrv}
\bibliography{bib}

\end{document}